\documentclass[a4paper,reqno,11pt]{amsart}
\usepackage{amssymb}
\usepackage{amsmath}
\usepackage{mathrsfs}
\usepackage[usenames]{color}
\usepackage[colorlinks=true, citecolor=blue, anchorcolor=red]{hyperref}
\usepackage{enumerate}

\numberwithin{equation}{section}

\allowdisplaybreaks

\begin{document}

\newtheorem{theorem}{Theorem}[section] 
\newtheorem{proposition}[theorem]{Proposition}
\newtheorem{corollary}[theorem]{Corollary}
\newtheorem{lemma}[theorem]{Lemma}

\theoremstyle{definition}
\newtheorem{assumption}[theorem]{Assumption}
\newtheorem{definition}[theorem]{Definition}

\theoremstyle{definition} 
\newtheorem{remark}[theorem]{Remark}
\newtheorem{remarks}[theorem]{Remarks}
\newtheorem{example}[theorem]{Example}
\newtheorem{examples}[theorem]{Examples}
\newenvironment{pf}%
{\begin{sloppypar}\noindent{\bf Proof.}}%
{\hspace*{\fill}$\square$\vspace{6mm}\end{sloppypar}}
\def\bA{{\bm A}}
\def\bB{{\bm B}}
\def\mE{{\mathbb E}}
\def\mK{{\mathbb K}}
\def\hmE{{\widehat{\mathbb E}}}
\def\mEp{{\mathbb E}_{\phi}}
\def\mFp{{\mathbb F}_{\phi}}
\def\mEpp{{\mathbb E}_{\phi,\mP}}
\def\tmEp{\widetilde{\mathbb E}_{\phi}}
\def\tmEpp{\widetilde{\mathbb E}_{\phi,P}}
\def\mFpp{{\mathbb F}_{\phi,\mP}}
\def\tPhi{\widetilde{\Phi}}
\def\mF{{\mathbb F}}
\def\mG{{\mathbb G}}
\def\mX{{\mathbb X}}
\def\mP{{\mathbb P}}
\def\db{\|}
\def\r{Nr}
\def\R{{\mathbb R}}
\def\bR{{\mathbb R}}
\def\N{{\mathbb N}}
\def\C{{\mathbb C}}
\def\Q{{\mathbb Q}}
\def\mP{{\mathbb P}}
\def\Z{{\mathbb Z}}
\def\mH{\mathbb H}
\def\mA{\mathbb A}
\def\mT{\mathbb T}
\def\bX{\mathbb X}
\def\bY{\mathbb Y}
\def\bZ{\mathbb Z}
\def\D{{\mathcal D}}
\def\cB{{\mathcal B}}
\def\E{{\mathcal E}}
\def\cF{{\mathcal F}}
\def\cA{{\mathcal A}}
\def\cH{{\mathcal H}}
\def\G{{\mathcal G}}
\def\B{{\mathcal B}}
\def\I{{\mathcal I}}
\def\M{{\mathcal M}}
\def\O{{\mathcal O}}
\def\S{{\mathcal S}}
\def\cT{{\mathcal T}}
\def\cP{{\mathcal P}}
\def\L{{\mathcal L}}
\def\cK{{\mathcal K}}
\def\cJ{{\mathcal J}}
\def\cS{{\mathcal S}}
\def\bH{{\bf H}}
\def\bP{{\bf P}}
\def\bQ{{\bf Q}}
\def\bE{{\bf E}}
\def\bT{{\bf T}}
\def\W{W}
\def\Be{L_\infty}
\def\cR{{\mathcal R}}
\def\eps{\varepsilon}
\def\3{{\ss}}
\def\slim{s-\lim_}
\def\capa{{\mathrm{Cap}}}
\def\supp{{\mathrm{supp}}}
\def\esssup{{\mathrm{ess\,sup}}}
\def\absconv{{\mathrm{absconv}}}
\def\dom{{\mathrm{dom}}}
\def\loc{\mathrm{loc}}
\def\hs{half-space}
\def\HIC{$\HH^\infty$-calculus}
\def\BIP{{\mathrm{BIP}}}
\def\BUC{{\mathrm{BUC}}}
\def\BC{{\mathrm{BC}}}
\def\MR{{\mathcal{MR}}}
\def\const{{\mathrm{const\,}}}
\def\Re{{\mathrm{Re}}}
\def\re{{\mathrm{Re}}}
\def\Im{{\mathrm{Im}}}
\def\im{{\mathrm{Im}}}
\def\dd{{\mathrm d}}
\def\graph{{\mathrm{graph}}}
\def\e{{\mathrm{e}}}
\def\id{{\mathrm{id}}}
\def\sb{{\mathrm{sb}}}
\def\FM{{\mathrm{FM}}}
\def\hperp{{^{_\perp}}}
\def\HIC{{$H^\infty$-calculus}}
\def\hW{\widehat{W}}
\def\hu{\hat{u}}
\def\hv{\hat{v}}
\def\hp{\hat{p}}
\def\hsigma{\hat{\sigma}}
\def\hf{\hat{f}}
\def\hh{\hat{h}}
\def\hg{\hat{g}}
\def\heta{\hat{\eta}}
\def\dR{\dot{\R}}
\def\tu{\tilde{u}}
\def\tc{\tilde{c}}
\def\tp{\tilde{p}}
\def\tf{\tilde{f}}
\def\th{\tilde{h}}
\def\tg{\tilde{g}}
\def\tv{\tilde{v}}
\def\ta{\tilde{a}}
\def\ty{\widetilde{y}}
\def\bv{\bar{v}}
\def\bw{\bar{w}}
\def\tsigma{\tilde{\sigma}}
\def\hphi{\hat{\phi}}
\newcommand{\essinf}[1]{{\mathrm{ess}}\!\inf_{\!\!\!\!\!\!\!\!\! #1}}
\newcommand{\fn}{\footnote}
\def\mdt{mixed derivative theorem}
\def\div{{\mathrm {div\,}}}
\def\bsigma{\bar{\sigma}}
\def\brho{\bar{\rho}}
\def\rcv{\W^1_p(J;L^p(\R^{n+1}_+))
        \cap L_p(J;\W^2_p(\R^{n+1}_+))}
\def\rcvtp{\W^1_p(J;L^p(\dR^{n+1}))
        \cap L_p(J;\W^2_p(\dR^{n+1}))}
\def\rcs{W^{3/2-1/2p}_p(J;L_p(\R^n))
        \cap \W^1_p(J;W^{1-1/p}_p(\R^n))
        \cap L_p(J;W^{2-1/p}_p(\R^n))}
\def\rcf{L_p(J;L_p(\R^{n+1}_+))}
\def\rcftp{L_p(J;L_p(\dR^{n+1}))}
\def\rch{W^{1/2-1/2p}_p(J;L_p(\R^n))
        \cap L_p(J;W^{1-1/p}_p(\R^n))}
\def\rcvi{W^{2-2/p}_p(\R^{n+1}_+)}
\def\rcvitp{W^{2-2/p}_p(\dR^{n+1})}
\def\rcsi{W^{2-2/p}_p(\R^n)}
\newcommand{\ab}{&\hskip-2mm}

\def\en{{\talloblong}}
\def\hookd{\stackrel{_d}{\hookrightarrow}}
\def\hook{{\hookrightarrow}}
\def\vt{{\vartheta}}
\def\ovt{{\overline{\vartheta}}}
\def\THE{tornado-hurricane equations}
\def\THO{tornado-hurricane operator}
\def\NSE{Navier-Stokes equations}
\def\SO{Stokes operator}
\def\HHP{Helmholtz projection}
\def\HHD{Helmholtz decomposition}
\def\HOL{\mathrm{HOL}}
\def\la{{\langle}}
\def\ra{{\rangle}}
\def\vphi{{\varphi}}
\def\vdp{{\{k:\ \alpha^v_k=0\}}}
\def\vndp{{\{k:\ \alpha^v_k\neq0\}}}
\def\tdp{{\{k:\ \alpha^\vt_k=0\}}}
\def\tndp{{\{k:\ \alpha^\vt_k\neq0\}}}
\def\sD{{\mathscr D}}
\def\hsD{{\widehat{\mathscr D}}}
\def\sL{{\mathscr L}}
\def\sR{{\mathscr R}}
\def\sT{{\mathscr T}}
\def\sLis{{\mathscr L}_{is}}
\def\PPr{{\PP_{\!\rho}}}
\def\ou{{\overline{u}}}
\def\oq{{\overline{q}}}
\def\oU{{\overline{U}}}
\def\th{{T\!H}}
\def\ttau{{\tilde \tau}}
\def\hH{{\widehat{H}}}
\def\cD{{\mathcal D}}
\def\vp{{\varphi}}
\def\Hic{{\mathcal{H}^\infty}}

\newcommand{\ord}{\operatorname{ord}}

\hyphenation{Lipschitz}


\sloppy
\title[Fluid-structure interaction]
{$L^p$-theory for a fluid-structure interaction model}

\author[R. Denk]{Robert Denk}
\address{Fachbereich Mathematik und Statistik,  Universit\"at Konstanz,
78457 Konstanz, Germany}
\email{robert.denk@uni-konstanz.de}

\author[J. Saal]{J\"urgen Saal}
\address{Mathematisches Institut, Angewandte Analysis\\
         Heinrich-Heine-Uni\-ver\-sit\"at D\"usseldorf\\
         40204 D\"usseldorf, Germany}
\email{juergen.saal@hhu.de}

\date{September 20, 2019}
\thispagestyle{empty}
\parskip0.5ex plus 0.5ex minus 0.5ex

\begin{abstract}
We consider a fluid-structure interaction model for an incompressible fluid where the elastic response of the free boundary is given by a damped Kirchhoff plate model. Utilizing  the Newton polygon approach,  we first prove
maximal regularity in $L^p$-Sobolev spaces  for a linearized version.
Based on this, we show existence and uniqueness of the  strong solution of the nonlinear system
for small data.
\end{abstract}

\subjclass{Primary 35Q30; Secondary 74F10, 76D05, 35K59}
\keywords{Fluid-structure interaction, maximal regularity, Newton polygon}


\maketitle

\section{Introduction and main result}\label{secintro}

We consider the system
\begin{equation}
	\label{fsi}
	\left.
	\begin{array}{rcll}
		\rho (\partial_t u + (u \cdot \nabla) u)) -
		\mbox{div}\,T(u,\,q) & = & 0,
		& \quad t > 0,\    x \in \Omega(t), \\[0.5em]
		                                                   \mbox{div}\,u & = & 0,                           & \quad t > 0,\    x \in \Omega(t), \\[0.5em]
		                            u & = & V_\Gamma, & \quad t \geq 0,\ x \in \Gamma(t), \\[0.5em]
		                           \frac1{\nu\cdot e_n}e_n^\tau T(u,\,q)\nu& = &
					   \phi_\Gamma,  & \quad t \geq 0,\ x \in \Gamma(t), \\[0.5em]
		                                \Gamma(0) = \Gamma_0,
						\quad
						V_\Gamma(0)=V_0,
						\quad u(0) & = & u_0,
						& \quad x \in \Omega(0),
	\end{array}
	\right\}
\end{equation}
which represents a (one-phase) fluid-structure interaction model.
The fluid with density $\rho>0$ and viscosity $\mu>0$ occupies at a time $t \geq 0$ the region $\Omega(t) \subseteq
\bR^n$ with boundary $\Gamma(t)=\partial\Omega(t)$.
Furthermore,
we assume the fluid to be incompressible, and we assume the stress to be given as
\begin{equation*}
	T(u,\,q) = 2 \mu D(u) - q, \qquad D(u) = {\textstyle
	\frac{1}{2}} (\nabla u + (\nabla u)^{\tau}).
\end{equation*}
The unknowns in the model are the velocity $u$, the pressure $q$ and the interface $\Gamma$.
We denote by $\nu$ the exterior unit normal field at $\Gamma$, by
$V_\Gamma$ the velocity of the boundary $\Gamma$, and by $e_j$ the
$j$-th standard basis vector in $\R^n$, i.e.\ $e_n=(0,\cdots,0,1)$.

The function $\phi_\Gamma$ describes the elastic response at $\Gamma$ which is given by a damped Kirchhoff-type plate model.
Throughout the paper we assume that $\Gamma$ is given as a graph of a
function $\eta:\R_+\times\R^{n-1}\to \R$, that is
\begin{equation}\label{gammagraph}
	\Gamma(t) = \Big\{(x',\eta(t,x'));\  x'\in
	\R^{n-1}\Big\}, \quad t \geq 0,
\end{equation}
and that $\Gamma(t)$ is sufficiently flat. Thus $\Omega(t)$ is a perturbed upper half-plane.In these coordinates, the elastic response is given  as
\begin{equation}\label{phigamma}
	\phi_\Gamma=m(\partial_t,\partial')\eta
	:=\partial_{t}^2\eta+\alpha(\Delta')^2\eta-\beta\Delta'\eta
	-\gamma\partial_t\Delta'\eta
\end{equation}
for $\alpha,\gamma>0$, $\beta\in\R$,
 where $\Delta'$ stands for the Laplacian in $\R^{n-1}$.
Finally, the initial configuration and velocity of the interface resp.\
the initial fluid velocity are given by $\Gamma_0$ and $V_0$ resp.\
$u_0= (u_0',u_0^n)$. We remark that in the formulation of the boundary
conditions in lines 3 and 4 of \eqref{fsi}, one has to take into account
that the Kirchhoff plate model is formulated in a Lagrangian setting,
whereas for the fluid an Eulerian setting is used. This is discussed in
more detail in the beginning of Section 2.

The symbol of $m(\partial_t,\partial')$ is given as
\[
	m(\lambda,\xi')=
	\lambda^2+\alpha|\xi'|^4+\beta|\xi'|^2
	+\gamma\lambda|\xi'|^2,\quad \lambda\in\C,\ \xi'\in\R^{n-1},
\]
which vanishes if
\[
	\lambda=-\frac{\gamma|\xi'|^2}{2}\pm\sqrt{\frac{\gamma^2|\xi'|^4}{4}-\alpha
	|\xi'|^4-\beta|\xi'|^2}.
\]
For $\gamma>0$ , the roots of $m(\cdot,\xi')$ lie in some sector which is a subset of $\{\lambda\in\C: \Re\lambda <0\}$. This indicates that the term
$-\gamma\partial_t\Delta'\eta$ in $\phi_\Gamma$ parabolizes the problem. Physically, one also speaks of structural damping of the plate.

We notice that basically the same results as proved in this note
can be expected by considering
layer like domains or rectangular type domains with periodic lateral
boundary conditions. For simplicity, however, we restrict the approach
given here to the just introduced geometry.

Model (\ref{fsi}) was introduced in \cite{Quarteroni2000} in connection
to applications to cardiovascular systems. In the 2D case, this system was investigated in \cite{batak2017} in the $L^2$-setting.
In fact, in \cite[Proposition~3.12]{batak2017} it is proved that
the linear operator associated to (\ref{fsi}) generates an analytic
$C_0$-semigroup in a suitable Hilbert space setting. This exhibits the
parabolic character of the problem. Therefore, it is reasonable to
consider $L^p$-theory for the system (\ref{fsi}) which is the main
purpose of this note.

Alternative approaches to system (\ref{fsi}) in the $L^2$-setting
also for the hyperbolic-parabolic case, i.e.\ $\gamma=0$,
are given, e.g., in \cite{cdeg2005,grandmont2008,Lengeler-Ruzicka14,Lengeler,Muha-Canic15},
concerning weak solutions and, e.g.,
in \cite{bdav2004,Coutand-Shkoller06,lequeurre2011,lequeurre2013}
concerning (local) strong solutions. A more recent
approach in an
two-dimensional $L^2$-framework concerning global strong solutions
is presented in \cite{grahil2016}. In the present paper, we develop an $L^p$-approach in general dimension for  system (\ref{fsi}). We show the existence of strong solutions for small data and give a precise description of the maximal regularity spaces for the unknowns. More precisely, we prove the following main result for (\ref{fsi}).
\begin{theorem}\label{main}
Let $n\ge 2$, $p\ge (n+2)/3$, $T>0$, and $J=(0,T)$.
Assume that
\[
	\|u_0\|_{W^{2-2/p}_p(\Omega(0))}
	+\|\eta_0\|_{W^{5-3/p}_p(\R^{n-1})}
	+\|\eta_1\|_{W^{3-3/p}_p(\R^{n-1})}
	<\kappa,
\]
where $\Gamma_0=\graph(\eta_0)$ and $V_0=\graph(\eta_1)$,
for some $\kappa>0$.
Then,
there exists a unique solution $(u,q,\Gamma)$ of system (\ref{fsi})
such that $\Gamma=\graph(\eta)$
and such that
\begin{align*}
	u&\in H^1_p(J;L^p(\Omega(t)))\cap
	L^p(J;H^2_p(\Omega(t))),\\
	q&\in L^p(J;\dot{H}^1_p(\Omega(t))),\\
	\eta&\in \mE_\eta :=W_{p}^{9/4 - 1/(4p)}(J; L^p(\R^{n-1}))
	\cap H_{p}^2(J; W_p^{1-1/p}(\R^{n-1}))\\
  & \quad \cap L^p(J;W_p^{5-1/p}(\R^{n-1})),
\end{align*}
provided that $\kappa=\kappa(T)$ is small enough
and that the following compatibility conditions are satisfied:
\begin{enumerate}
\item $ \div u_0 = 0$,
\item if  $p>\tfrac32$, then $u_0'|_{\Gamma_0} = 0$
and $u_0^n|_{\Gamma_0} - \eta_1=0$ almost everywhere,
\item there exists an $\eta_*\in\mE_\eta$ with $\eta_*|_{t=0} = \eta_0$, $\partial_t\eta_*|_{t=0}=\eta_1$ and
    \[
      \partial_t \eta_*\in H_{p}^1(\R_+; {}_0\dot
      H_p^{-1}(\R^n_+)),
    \]
    where
\[
	\partial_t \eta_*(\phi) := - \int_{\R^{n-1}} \partial_t \eta_*\phi
	dx',\quad \phi\in\dot H_{p'}^1(\R^n_+).
\]
\end{enumerate}
The solution depends continuously on the data.
\end{theorem}

\begin{remark}
  a) The compatibility conditions (1)--(3) are natural in the sense that they are also necessary for the existence of a strong solution. Condition (3) appears in a similar way for the two-phase Stokes problem, see, e.g., \cite{Pruess-Simonett16}, Section~8.1.

  b) We remark that the maximal regularity space $\mE_\eta$ for $\eta$
  describing the boundary is not a standard space. It is given as an
  intersection of three Sobolev spaces. This is due to the fact that the
  symbol of the complete system has an inherent inhomogeneous structure,
  and therefore the Newton polygon method is the correct tool
  to show maximal regularity. For the details, see Section~3 below.

  c) We note that in the physically relevant situations $n=2$ and $n=3$,
  the case $p=2$ is included. This might be of importance when
  considering the singular limit $\gamma\to0$ for vanishing
  damping of the plate.

  d) We formulated the result in the form of existence for fixed time and small data. By the same methods, one can also show short time existence for
 arbitrarily large data.
\end{remark}

The proof of Theorem~\ref{main} is based on several ingredients: First,
we transform the system to a fixed domain and consider the linearization
of the transformed system. By an application of the Newton polygon
approach (see, e.g., \cite{Denk-Kaip13} and \cite{Denk-Saal-Seiler08}),
we obtain maximal regularity for the linearized system. To deal with the nonlinearities, we employ embedding results on anisotropic Sobolev spaces given in \cite{koehnesaal}.

\section{The transformed system}\label{sectrans}

We start with a short discussion of the boundary conditions, where the Eulerian approach for the fluid has to be coupled with the Lagrangian description for the plate (see also  \cite{Lengeler-Ruzicka14} and \cite{grandmont2008}).  Let $\Gamma$ be given as in (\ref{gammagraph}) and assume that $\eta$ is sufficiently smooth. Following the Kirchhoff plate model, in-plate motions are ignored, and the velocity of the plate at the point $(x',\eta(t,x'))^\tau$ is parallel to the  vertical direction and given by $(0,\partial_t \eta(t,x'))^\tau = \partial_t   \eta(t,x') e_n$. As the fluid is assumed to adhere to the plate, we have no-slip boundary conditions for the fluid, and the equality  of the velocities yields the first boundary condition
\begin{equation}
  \label{eq1-1}
  u(t,x',\eta(t,x')) = \partial_t   \eta(t,x') e_n \quad (t>0, \, x'\in \R^n).
\end{equation}
The exterior normal at the point $(x',\eta(t,x'))$ of the boundary $\Gamma(t)$ is given by
\[
	\nu=\nu(t,x')=\frac1{\sqrt{1+|\nabla'\eta(t,x')|^2}}
	\left(
	\begin{array}{c}
	\nabla'\eta(t,x')\\
	-1
	\end{array}
	\right).
\]
We define the transform of variables
$$
\theta : J \times \R^{n}_+ \to
  \bigcup_{t \in J} \{t\} \times \Omega(t), \ (t,x',x_n)\mapsto\theta(t,x',x_n):=
  (t,x',x_n+\eta(t,x')).
$$
Obviously we have $\theta^{-1}(t,x',y)=(t,x',y-\eta(t,x'))$. As it was discussed in \cite{Lengeler-Ruzicka14}, Section~1.2,
the  force $F$ exerted by the
fluid on the boundary is given by the evaluation of the stress tensor at the deformed
boundary in the direction of the inner normal $-\nu(t,x')$. More precisely, we obtain (\cite{Lengeler-Ruzicka14}, Eq. (1.4))
\[ F = - \sqrt{1+|\nabla'\eta(t,x')}\; e_n^\tau  (T(u,q)\circ \theta(t,x)) \nu(t,x'). \]
As $\sqrt{1+|\nabla' \eta|^2} = - \nu(t,x')\cdot e_n$, the equality of the forces gives the second boundary condition
\begin{equation}
  \label{eq1-2}
  \begin{aligned}
  \frac{1}{\nu(t,x')\cdot e_n} \; e_n^\tau [ T(u,q) ] (t,x',\eta(t,x')) \, \nu(t,x') & = [m(\partial_t,\partial')\eta] (t,x')\\ & \qquad (t>0,x'\in\R^{n-1}).
  \end{aligned}
\end{equation}
Conditions \eqref{eq1-1} and \eqref{eq1-2} are the precise formulation of the boundary conditions in \eqref{fsi}.

To solve the problem \eqref{fsi}, we first note that by a re-scaling argument we may assume that
$\rho=\mu=1$ for the density $\rho$ and viscosity $\mu$ from now on.
Next, we transform the problem (\ref{fsi}) to a problem on the fixed
half-space $\R^n_+$, using the above transformation $\theta$. To this end, we set $J:=(0,T)$ and write
$x=(x',x_n)\in\R^n_+$ with $x'\in\R^{n-1}$. With the
corresponding meaning we write $v'$, $\nabla'$, etc.
The pull-back is then defined as
\[
	v:=\Theta^\ast u:= u\circ\theta, \qquad
	p:= \Theta^\ast q := q\circ\theta,
\]
and correspondingly the push-forward as
\[
	u:=\Theta_\ast v:= v\circ\theta^{-1}, \qquad
	q:= \Theta_\ast p := p\circ\theta^{-1}.
\]
We also set $\Gamma_0=\Gamma(0)=\{(x',\eta_0(x'));\ x'\in\R^{n-1}\}$ and
$V_0=V_\Gamma(0)=(0,\eta_1(\cdot))^\tau$.

Applying the transform of variables to (\ref{fsi}) leads to the following
quasilinear system for $(v,p,\eta)$:
\begin{equation}
	\label{tfsi}
	\begin{array}{rclll}
		                               \partial_t v -
					       \Delta v + \nabla
					       p& = & F_v(v,p,\eta)
					       & \mbox{in} & J
					       \times \R^n_+, \\
		                                               \div v &
							       = &
		G(v,\eta)&  \mbox{in} & J \times \R^n_+, \\
		v'&=&0
		& \mbox{on} & J \times \R^{n-1},      \\
		\partial_t\eta - v^n & = & 0
		&  \mbox{on} & J \times \R^{n-1},      \\
		  -2 \partial_n v^n + p
		  -m(\partial_t,\partial')\eta& = & H_\eta(v,\eta)
		  &  \mbox{on} & J \times \R^{n-1},      \\
		          v|_{t=0} &=& v_0  &\mbox{in} & \R^n_+,\\
			  \eta|_{t=0}&= & \eta_0
			  & \mbox{in} & \R^{n-1},\\
			  \partial_t\eta|_{t=0} & = & \eta_1
			  & \mbox{in} & \R^{n-1}.
	\end{array}
\end{equation}
The non-linear right-hand sides are given as
\begin{equation*}
	\begin{array}{rcl}
		                 F_v(v,p,\eta) & = & (
				 \partial_t \eta - \Delta' \eta)
				 \partial_n v - 2  (\nabla' \eta
				 \cdot \nabla') \partial_n v +
				 |\nabla' \eta|^2 \partial^2_n v \\[0.5em]
			                             &   & \quad -
							   (v \cdot
							  \nabla) v +
							   (v' \cdot
							  \nabla'\eta) \partial_n v
							  +
							  (\nabla'\eta,0)^\tau\partial_n
							  p,               \\[0.5em]
		                         G(v,\eta) & = & \nabla'\eta
					 \cdot \partial_n v',                                                                                                                        \\[0.5em]
		                     H_\eta(v,\eta) & = & -
				     \nabla'\eta \cdot \partial_n v'
				     - \nabla'\eta \cdot \nabla' v^n.                                          \end{array}
\end{equation*}

\section{The linearized system}

The aim of this section is to derive maximal regularity for the
 linearized system
\begin{equation}
	\label{linfsi}
	\begin{array}{rclll}
		                                \partial_t v -
					       \Delta v + \nabla
					       p& = & f_v
					       & \mbox{in} & \R_+
					       \times \R^n_+, \\
		                                               \div v &
							       = &
		g&  \mbox{in} & \R_+ \times \R^n_+, \\
		v'&=&0
		& \mbox{on} & \R_+ \times \R^{n-1},      \\
		\partial_t\eta - v^n & = & 0
		&  \mbox{on} & \R_+ \times \R^{n-1},      \\
		  -2 \partial_n v^n + p
		  -m(\partial_t,\partial')\eta& = & f_\eta
		  &  \mbox{on} & \R_+ \times \R^{n-1},      \\
		          v|_{t=0} &=& v_0  &\mbox{in} & \R^n_+,\\
			  \eta|_{t=0}&= & \eta_0
			  & \mbox{in} & \R^{n-1},\\
			  \partial_t\eta|_{t=0} & = & \eta_1
			  & \mbox{in} & \R^{n-1}.
	\end{array}
\end{equation}

We will consider this system in Sobolev spaces with exponential weight with respect to the time variable. Let $\rho\in\R$ and $X$ be a Banach space. For $u\in L^p(\R_+,X)$, we define $\Psi_\rho$ as the multiplication operator with $e^{-\rho t}$, i.e. $\Psi_\rho u(t) := e^{-\rho t}u(t),\; t\in\R_+$. The spaces with exponential weights are defined by
\begin{align*}
  H_{p,\rho}^s(\R_+,X) & := \Psi_{-\rho}(H_p^s(\R_+,X)),\\
   W_{p,\rho}^s(\R_+,X) & := \Psi_{-\rho}(W_p^s(\R_+,X))
\end{align*}
with canonical norms $\|u\|_{H_{p,\rho}^s(\R_+,X) } := \| \Psi_\rho u\|_{H_p^s(\R_+,X)}$ and
$\|u\|_{W_{p,\rho}^s(\R_+,X) } := \| \Psi_\rho u\|_{W_p^s(\R_+,X)}$. For $\rho\ge 0$ and $s>0$, we define ${}_0 H_{p,\rho}^s(\R_+,X)$ and ${}_0 W_{p,\rho}^s(\R_+,X)$ analogously, replacing $H_p^s$ and $W_p^s$ by ${}_0H_p^s$ and ${}_0W_p^s$, respectively. For mapping properties and interpolation results under the condition that $X$ is a UMD space, we refer, e.g., to \cite{Denk-Saal-Seiler08}, Lemma~2.2.
We also make use of homogeneous spaces, e.g., for $\Omega\subset\R^n$ we set
\[
	\dot H_p^1(\Omega):=\{v\in L^1_{\loc}(\Omega):\ \nabla v\in
	L^p(\Omega)\}
\]
and $\dot
H_{p,0}^1(\Omega):=\overline{C^\infty_c(\Omega)}^{\|\nabla\cdot\|_p}$.
The corresponding dual spaces are defined as
\[
	\dot H^{-1}_p(\Omega):=\bigl(\dot H^1_{p,0}(\Omega)\bigl)'
	\quad\text{and}\quad
	\dot H^{-1}_{p,0}(\Omega):=\bigl(\dot H^1_{p}(\Omega)\bigl)',
\]
see \cite{Pruess-Simonett16} Section~7.2. The homogeneous
Sobolev-Slobodeckii spaces over $\R^n$
are defined as usual \cite{Triebel15} and we have
\[
	\dot W^{s}_p(\R^n)=\dot B^s_{pp}(\R^n)
\]
for $1<p<\infty$, $n\in\N$, and $s\in\R\setminus\Z$,
where the latter one denotes the homogeneous Besov space.

In the following, we denote the time trace $u\mapsto \partial_t^k u|_{t=0}$ by $\gamma_k^t$ and the trace to the boundary $u\mapsto \partial_n^k u|_{\R^{n-1}}$ by $\gamma_k$.
The solution $(v,p,\eta)$ of \eqref{linfsi} will belong to the spaces
\begin{align*}
  v\in \mE_v & := H_{p,\rho}^1(\R_+;L^p(\R^n_+))\cap L^p_\rho(\R_+;H_p^2(\R^n_+)),\\
  p\in\mE_p & := L^p_\rho(\R_+; \dot H_p^1(\R^n_+)),\\
  \eta\in\mE_\eta & := W_{p,\rho}^{9/4 - 1/(4p)}(\R_+; L^p(\R^{n-1}))\cap H_{p,\rho}^2(\R_+; W_p^{1-1/p}(\R^{n-1}))\\
  & \quad \cap L^p_\rho(\R_+;W_p^{5-1/p}(\R^{n-1})).
\end{align*}
The function spaces for the right-hand side of \eqref{linfsi} are given by
\begin{align*}
  f_v \in \mF_v & := L^p_\rho(\R_+; L^p(\R^n_+)),\\
  g\in \mF_g & :=  H_{p,\rho}^1(\R_+;\dot H_p^{-1}(\R^n_+))\cap L^p_\rho(\R_+;H_p^1(\R^n_+)),\\
 f_\eta\in \gamma_0\mE_p & := L^p_\rho(\R_+; \dot W_p^{1-1/p}(\R^{n-1})).
\end{align*}
By trace results with respect to the time trace, the spaces for the initial values are given by
\begin{align*}
  v_0 \in \gamma_0^t\mE_v &:= W_p^{2-2/p}(\R^n_+),\\
  \eta_0  \in \gamma_0^t\mE_\eta & := W_p^{5-3/p}(\R^{n-1}),\\
  \eta_1  \in \gamma_1^t\mE_\eta & := W_p^{3-3/p}(\R^{n-1}).
\end{align*}
We will also need the following compatibility conditions:
\begin{enumerate}[(C1)]
\item $ \div v_0 = g|_{t=0}$ in $\dot H_p^{-1}(\R^{n}_+)$.
\item If  $p>\tfrac32$, then $v_0'|_{\R^{n-1}} = 0$ almost everywhere in $\R^{n-1}$.
\item If  $p>\tfrac32$, then $v_0^n|_{\R^{n-1}} - \eta_1=0$ almost everywhere in $\R^{n-1}$.
\item There exists an $\eta_*\in\mE_\eta$ with $\eta_*|_{t=0} = \eta_0$, $\partial_t\eta_*|_{t=0}=\eta_1$ and
    \begin{equation}
      \label{eq2-6}
      (g,\partial_t \eta_*)\in H_{p,\rho}^1(\R_+; \dot H_{p,0}^{-1}(\R^n_+)).
    \end{equation} Here we define
 \[ (g,\partial_t \eta_*)(\phi) := \int_{\R^n_+} g\phi dx - \int_{\R^{n-1}} \partial_t \eta_*\phi dx'\]
       for $\phi\in\dot H_{p'}^1(\R^n_+)$.
       Additionally, we have $(g|_{t=0},\eta_1)=(g|_{t=0},
       v_0^n|_{\R^{n-1}})$ in $\dot H_{p,0}^{-1}(\R^n_+)$.
\end{enumerate}
We remark that only \eqref{eq2-6} is an additional condition, as it was shown in \cite{Denk-Saal-Seiler08}, Theorem~4.5,  that for every $\eta_0\in\gamma_0^t\mE_\eta$ and $\eta_1\in\gamma_1^t\mE_\eta$ there exists an $\eta_*\in \mE_\eta$ with $\eta_*|_{t=0} = \eta_0$ and $\partial_t\eta_*|_{t=0} = \eta_1$.

The main result of this section is the following maximal regularity result.

\begin{theorem}
  \label{2.2}
  Let $p>1$, $p\neq3/2$. Then there exists a $\rho_0>0$ such that for every $\rho\ge \rho_0$, system \eqref{linfsi} has a unique solution $(v,p,\eta)\in \mE_v\times\mE_p\times\mE_\eta$ if and only if the data $f_v,g,f_\eta,v_0,\eta_0,\eta_1$ belong to the spaces above and satisfy the compatibility conditions \upshape{(C1)--(C4)}. The solution depends continuously on the data.
\end{theorem}

The proof of this theorem will be done in several steps and follows from Subsections~\ref{sec_necessity}--\ref{sec_uniqueness}.

\subsection{Necessity}
\label{sec_necessity}
Let $(v,p,\eta)\in\mE_v\times \mE_p\times \mE_\eta$ be a solution of \eqref{linfsi}. By standard continuity and  trace results, the right-hand sides  $f_v,$ and $g$ as well as the time trace $v_0$ belong to the spaces above. Noting that $\div\colon L^p(\R^n_+)\to \dot H_p^{-1}(\R^n_+)$ is continuous, we have $g = \div u \in H_p^1(\R_+;\dot H_p^{-1}(\R^n_+)) \subset C([0,\infty);\dot H_p^{-1}(\R^n_+))$, and as for all $p>1$ we also have $v_0\in W_p^{2-2/p}(\R^n_+)\subset L^p(\R^n_+)$, we obtain the compatibility condition (C1) for all $p>1$ (see also \cite{Pruess-Simonett16}, Theorem~7.2.1).

For $f_\eta$ we have $\mE_\eta\subset H_{p,\rho}^2(\R_+;W_p^{1-1/p}(\R^{n-1})$ by the mixed derivative theorem (see, e.g., \cite{Denk-Saal-Seiler08}, Lemma~4.3), and therefore
\[\partial_t^2\eta \in L^p_\rho(\R_+; W_p^{1-1/p}(\R^{n-1})\subset \gamma_0\mE_p.\]
It is easy to see that the other terms of $m(\partial_t,\partial')\eta$ belong to the same space.
By standard trace results, we also obtain $\gamma_0 u\in \gamma_0\mE_p$. Concerning the pressure, we remark that $\gamma_0\colon \dot H_p^1(\R^n_+) \to \dot W_p^{1-1/p}(\R^{n-1})$ is a retraction, see, e.g., \cite{Jawerth78}, Theorem 2.1, and therefore $\gamma_0 p\in \gamma_0\mE_p$. This yields $f_\eta\in \gamma_0\mE_p$. For the time traces of $\eta$, we can apply \cite{Denk-Saal-Seiler08}, Theorem~4.5 which gives $\eta_0\in\gamma_0^t\mE_\eta$ and $\eta_1\in \gamma_1^t\mE_\eta$.

If $p>\frac32$, then the boundary trace of $v_0$  exists in the space
$W_p^{2-3/p}(\R^{n-1})$. This yields the compatibility conditions (C2)
and (C3) as equality in the space $W_p^{2-3/p}(\R^{n-1})$, hence in
particular as equality almost everywhere.

To show (C4), we can set $\eta_* := \eta$.
 For $\phi\in \dot H_{p'}^1(\R^n_+)$, we obtain
\[ (g,\partial_t\eta)(\phi) = \int_{\R^n_+}  \div u \,\phi dx  - \int_{\R^{n-1}} u^n \phi dx' = -\int_{\R^n_+} u \cdot \nabla \phi dx \]
and therefore $(g, \partial_t\eta)\in H_{p,\rho}^1(\R_+; \dot H_{p,0}^{-1}(\R^n_+))$.
Setting $t=0$, we obtain $(g|_{t=0}, \eta_1) = (g|_{t=0}, v_0^n)$ as
equality in $\dot H_{p,0}^{-1}(\R^n_+)$.

\subsection{Reductions}
\label{sec_reduction}

We can reduce some part of the right-hand side of \eqref{linfsi} to zero by applying known results on the Stokes system. For this, let $(v^{(1)}, p^{(1)})\in \mE_v\times \mE_p$ be the unique solution of the Stokes problem in the half space
\begin{equation} \label{eq2-3}
	\begin{array}{rclll}
		                                \partial_t v^{(1)} -
					       \Delta v + \nabla
					       p^{(1)}& = & f_v
					       & \mbox{in} & \R_+
					       \times \R^n_+, \\
		                                               \div v^{(1)} &
							       = &
		g&  \mbox{in} & \R_+ \times \R^n_+, \\
		(v^{(1)})'&=&0
		& \mbox{on} & \R_+ \times \R^{n-1},      \\
		 (v^{(1)})^n & = & \partial_t\eta_*
		&  \mbox{on} & \R_+ \times \R^{n-1},      \\
		          v^{(1)}|_{t=0} &=& v_0  &\mbox{in} & \R^n_+.
	\end{array}
\end{equation}
The unique solvability of \eqref{eq2-3} follows from \cite{Pruess-Simonett16}, Theorem~7.2.1.
To show that this theorem can be applied, we remark in particular that the compatibility condition (e) in \cite[p.~324]{Pruess-Simonett16} holds because of (C4).

Let $\tilde v := v-v^{(1)} $,  $\tilde p := p - p^{(1)}$, and $\tilde\eta := \eta - \eta_*$. Then $(v,p,\eta)$ is a solution of \eqref{linfsi} if and only if $(\tilde v,\tilde p,\tilde \eta)$ is a solution of
\begin{equation}\label{eq2-4}
	\begin{array}{rclll}
		                                \partial_t \tilde v -
					       \Delta \tilde v + \nabla
					       \tilde p& = & 0
					       & \mbox{in} & \R_+
					       \times \R^n_+, \\
		                                               \div \tilde v &
							       = &
		0&  \mbox{in} & \R_+ \times \R^n_+, \\
		\tilde v'&=&0
		& \mbox{on} & \R_+ \times \R^{n-1},      \\
		 \tilde v^n -  \partial_t\tilde \eta& = &  0
		&  \mbox{on} & \R_+ \times \R^{n-1},      \\
		  -2 \partial_n \tilde v^n + \tilde p
		  -m(\partial_t,\partial')\tilde \eta& = & \tilde f_\eta
		  &  \mbox{on} & \R_+ \times \R^{n-1},      \\
		          \tilde v|_{t=0} &=& 0  &\mbox{in} & \R^n_+,\\
			  \tilde\eta|_{t=0}&= & 0
			  & \mbox{in} & \R^{n-1},\\
			  \partial_t\tilde\eta|_{t=0} & = & 0
			  & \mbox{in} & \R^{n-1}.
	\end{array}
\end{equation}
Here,
\begin{align*}
  \tilde f_\eta & := f_\eta + 2\partial_n (v^{(1)})^n
   - p^{(1)}  +  m(\partial_t, \partial') \eta_*.
\end{align*}
By the trace results in Subsection~\ref{sec_necessity},
we have $\tilde f_\eta \in \gamma_0\mE_p$.

\subsection{Solution operators for the reduced linearized problem}
In the following, we show solvability for the reduced problem \eqref{eq2-4}, omitting the tilde again. An application of the Laplace transform formally leads to the resolvent problem
\begin{equation}
	\label{resfsi}
	\begin{array}{r@{\ =\ }ll}
		\lambda v-\Delta v+\nabla p &   0&\text{in} \ \R^n_+,\\
	\mbox{div}\,v &   0&\text{in} \ \R^n_+,\\
		v'&0&\text{on} \ \partial\R^n_+,\\
		v^n- \lambda\eta&0&\text{on} \ \partial\R^n_+,\\
	-2\partial_n v^n+p-m(\lambda,\partial')\eta&f_\eta&\text{on} \
	\partial\R^n_+
	\end{array}
\end{equation}
with
\[
	m(\lambda,\partial')\eta=
	\lambda^2\eta+\alpha(\Delta')^2\eta-\beta\Delta'\eta
	-\gamma\lambda\Delta'\eta.
\]
We observe that the second and the third line of \eqref{resfsi} imply that
\[
	\partial_n v^n(\cdot,0)=-\nabla'\cdot v'(\cdot,0)=0.
\]
Hence the fifth line reduces to
\[
p-m(\lambda,\partial')\eta=f_\eta\quad\text{on}\ \partial\R^n_+.	
\]

Applying partial Fourier transform in
$x'\in\R^{n-1}$,
we obtain the following system of ordinary differential equations in
$x_n$ for the
transformed functions  $\hat v$, $\hat p$ and $\hat \eta$:
\begin{align*}
\omega^2 \hat v - \partial_n^2 \hat v + (i \xi',\partial_n)^\tau \hat p & =
0,
\quad x_n>0, \nonumber \\
i \xi \cdot \hat v' + \partial_n \hat v^n & = 0, \quad x_n>0, \\
\hv'&=0,\quad x_n=0,\\
\lambda \heta -\hat v^n & = 0, \quad x_n=0,\\
\hp-m(\lambda,|\xi'|)\heta&=\hf_\eta,\quad x_n=0,\\
\end{align*}
Here we have set $\omega:= \omega(\lambda,\xi'):=\sqrt{\lambda +
|\xi'|^2}$ and
\[
	m(\lambda,\xi'):=\lambda^2+\alpha|\xi'|^4+\gamma\lambda|\xi'|^2
	+\beta|\xi'|^2.
\]

Multiplying the first equation  with  $(i\xi',\partial_n)$ and
combing it with the second one yields
$(-|\xi'|^2+\partial_n^2)\hp = 0$ for $x_n>0$.
The only stable solution of this equation is given by
\begin{equation}\label{(2.7)}
  \hp(\xi',x_n) = \hp_0(\xi') e^{-|\xi'|x_n},
  \quad \xi' \in {\mathbb R}^{n-1}, \ x_n > 0.
\end{equation}
To solve the above system we employ the ansatz
\begin{align}
  \hat v'(\xi',x_n) & = -\int_0^\infty k_+(\lambda,\xi',x_n,s)
  i\xi'\hp(\xi',s)  ds + \hat \phi'(\xi') e^{-\omega x_n},\label{rd_eq1}\\
  \hat v^n(\xi',x_n) & = -\int_0^\infty k_-(\lambda,\xi',x_n,s)
  \partial_n \hp(\xi',s)ds + \hat \phi^n(\xi') e^{- \omega x_n}\label{rd_eq1a}
\end{align}
with
$$
  k_{\pm} (\lambda,\xi,x_n,s) := \frac{1}{2 \omega} \left( e^{-
  \omega|x_n-s|} \pm e^{-\omega (x_n+s)} \right).
$$
Here the traces $\hat p_0$ and $\hat\phi = (\hat \phi',
\hat \phi^n)^\tau$ still have to be determined.
The fact that $\div v=0$ enforces
\begin{equation}\label{divbed}
	i\xi'\cdot\hphi'(\xi')=\omega\hphi^n(\xi').
\end{equation}
The kinematic boundary condition instantly gives us
\begin{equation}\label{kinbc}
	\lambda\heta-\hphi^n=0.
\end{equation}
Next, by utilizing (\ref{(2.7)}), from the tangential boundary condition
we obtain
\[
	0=\hv'(\xi',0)=-\int_0^\infty \frac{e^{-\omega s}}{\omega}
  i\xi'\hp(\xi',s) ds + \hat \phi'(\xi'),
\]
which implies
\begin{equation}
\label{eq2-9}
	\frac{i\xi'}{\omega+|\xi'|}\hp_0=\omega\hphi'.
\end{equation}
Multiplying this with $i\xi'$ and employing the relations
(\ref{divbed}), (\ref{kinbc}) yields
\begin{equation}\label{tangbc}
	-\frac{|\xi'|^2}{\omega+|\xi'|}\hp_0
	=\omega^2\hphi^n=\lambda\omega^2 \heta.
\end{equation}
Plugging this into the last line of the transformed system we obtain
\begin{equation}\label{normaltrace}
	\left(\frac{\lambda\omega^2(\omega+|\xi'|)}{|\xi'|^2}
	+m(\lambda,\xi')\right)\heta=-\hf_\eta.
\end{equation}
This yields
\begin{equation}\label{eq2-8}
	\heta=-\frac{|\xi'|^2}{N_L(\lambda,|\xi'|)}\hf_\eta
\end{equation}
with
\[
	N_L(\lambda,|\xi'|)=|\xi'|^2m(\lambda,\xi')
	+\lambda\omega^2(\omega+|\xi'|).
\]

Formula \eqref{eq2-8} defines the solution operator for $\eta$ as a
function of $f_\eta$ on the level of its Fourier-Laplace transform. The
following result is based on the Newton polygon approach and shows that
the solution operator is continuous on the related Sobolev spaces. In
the following, we consider $(-\Delta')^{1/2}$ as an unbounded operator
in $L^p_\rho(\R_+; L^p(\R^{n-1})$, and define
$N_L(\partial_t,(-\Delta')^{1/2})$ by the joint $H^\infty$-calculus of
$\partial_t$ and $(-\Delta')^{1/2}$ (for details, we refer to, e.g.,
\cite{Denk-Saal-Seiler08}, Corollary~2.9). We will apply the Newton
polygon approach on the Bessel potential scale $H_p^s$ with respect to
time and on the Besov scale $B_{pp}^r$ with respect to space.

\begin{lemma}
  \label{2.3}
  a) There exists a $\rho_0>0$ such that for all $\rho\ge \rho_0$, the operator
  $ N_L(\partial_t, (-\Delta')^{1/2})\colon H_N \to L^p_\rho(\R_+;B_{pp}^{-1-1/p}(\R^{n-1}) ) $ is an isomorphism, where
  \begin{align*}
   H_N &  := {}_0 H_{p,\rho}^{5/2}(\R_+;B_{pp}^{-1-1/p}(\R^{n-1}))\cap {}_0H_{p,\rho}^2(\R_+;B_{pp}^{1-1/p}(\R^{n-1})) \\
  & \quad \cap L^p_\rho(\R_+; B_{pp}^{5-1/p}(\R^{n-1})).
  \end{align*}

  b) Let $\rho\ge \rho_0$. Then for every $f_\eta\in \gamma_0\mE_p$, we have
  \begin{align*}
    \eta & := \Delta' \big[ N_L(\partial_t, (-\Delta')^{1/2})\big]^{-1}  f_\eta \in \mE_\eta,\\
    \phi^n  & := \partial_t \eta \in \gamma_0\mE_v,\\
    p_0 & := f_\eta +  m(\partial_t,\partial')\eta \in \gamma_0\mE_p.
  \end{align*}
\end{lemma}

\begin{proof}
  a) We apply the Newton-polygon approach developed in
\cite{Denk-Saal-Seiler08}. Replacing $z=|\xi'|$,
the $r$-principle symbols, i.e., the leading terms
of $N_L$ associated to the
relation $\lambda\sim z^r$, are easily calculated as
\[
	P_r(\lambda,z)=
	\left\{
	\begin{array}{rl}
	\alpha z^6,& 0<r<2,\\
	m_0(\lambda,z)z^2,& r=2,\\
	\lambda^2 z^2,&2<r<4,\\
	\lambda^2 z^2+\lambda^{5/2},&r=4,\\
	\lambda^{5/2},& r>4,
	\end{array}
	\right.
\]
where $m_0=m$ for $\beta=0$, that is
\[
	m_0(\lambda,z):=\lambda^2+\alpha z^4+\gamma\lambda z^2.
\]
In other words, the associated Newton-polygon has the three relevant vertices $(6,0)$, $(2,2)$, and $(0,\frac52)$ and two relevant edges which again reflects
the quasi-homogeneity of $N_L$.

Now, let $\vp\in(0,\pi/2)$ and $\theta\in(0,\vp/4)$. For $r\neq 2$
we then obviously have
\begin{equation}\label{rsymbcond}
	P_r(\lambda,z)\neq 0\quad
	\left((\lambda,z)\in\Sigma_{\pi-\vp}\times\Sigma_\theta\right).
\end{equation}
For $r=2$ we deduce
\[
	P_2(\lambda,z)=0
	\quad\Leftrightarrow \quad \lambda=\frac{z^2}{2}
	\left(-\gamma\mp \sqrt{\gamma^2-4\alpha}\right).
\]
By the fact that $\gamma>0$ we see that
\[
	\vp_0:=\pi-\arg \left(-\gamma\mp \sqrt{\gamma^2-4\alpha}\right)
	<\frac{\pi}2.
\]
Thus, assuming $\vp\in(\vp_0,\pi/2)$ and
$\theta\in\left(0,(\vp-\vp_0)/4\right)$ we see that (\ref{rsymbcond})
is satisfied for all $r>0$.
This allows for the application of \cite[Theorem~3.3]{Denk-Saal-Seiler08} (setting $s=0$ and $r=-1-1/p$ in the notation of \cite{Denk-Saal-Seiler08}) which yields a).

b) We write
\[ \eta =   [ N_L(\partial_t, (-\Delta')^{1/2})\big]^{-1}\Delta'   f_\eta.\]
As $\Delta'$ is an isomorphism from $\dot H_p^{2+t}(\R^{n-1})$ to $\dot H_p^{t}(\R^{n-1})$ for each  $t\in\R$,
by real interpolation of these spaces  (see \cite{Jawerth78}, Lemma 1.1) we see that it is also an isomorphism
from $\dot B_{pp}^t(\R^{n-1})$ to $\dot B_{pp}^{t-2}(\R^{n-1})$ for each $t\in\R$. In particular, $\Delta' f_\eta \in L^p_\rho(\R_+; \dot B_{pp}^{-1-1/p}(\R^{n-1}))$. Using the fact that for $s<0$ the embedding $\dot B_{pp}^s (\R^{n-1}) \subset B_{pp}^s( \R^{n-1})$ holds (see \cite[p. 104, (3.339)]{Triebel14}, \cite[Section~3.1]{Triebel15}), we obtain the embedding
\[ L^p_\rho(\R_+; \dot B_{pp}^{-1-1/p}(\R^{n-1})) \subset L^p_\rho(\R_+;  B_{pp}^{-1-1/p}(\R^{n-1})).\]
An application of a) yields
\begin{align*}
 \eta & \in {}_0H_{p,\rho}^{5/2}(\R_+;B_{pp}^{-1-1/p}(\R^{n-1}))  \cap {}_0H_{p,\rho}^{2}(\R_+;B_{pp}^{1-1/p}(\R^{n-1}))\\
 & \quad \cap  L^p_\rho(\R_+; B_{pp}^{5-1/p}(\R^{n-1})).
 \end{align*}
Now, the mixed derivative theorem in mixed scales (see
\cite{Denk-Kaip13}, Proposition~2.7.6) implies
\begin{align*}
 {}_0H_{p,\rho}^{5/2}(\R_+;&B_{pp}^{-1-1/p}(\R^{n-1}))  \cap L^p_\rho(\R_+; B_{pp}^{5-1/p}(\R^{n-1}))\\
 &  \subset B_{pp,\rho}^{9/4 - 1/(4p)}(\R_+; L^p(\R^{n-1}))
\end{align*}
and we obtain $\eta\in \mE_\eta$.

For $u^n := \partial_t \eta$, we immediately get
\[ u^n  \in W_{p,\rho}^{5/4-1/(4p)}(\R_+; L^p(\R^{n-1}))\cap L^p_\rho(\R_+;W_p^{3-1/p}(\R^{n-1})\subset\gamma_0\mE_v.\]
Finally, the fact that $m(\partial_t, \partial')\eta \in \gamma_0\mE_p$ for $\eta\in \mE_\eta$ was already remarked in Subsection~\ref{sec_necessity}.
\end{proof}

Due to the last result, we obtain the existence of a solution $(v,p,\eta)$ of \eqref{resfsi}. For  $\eta$, $\phi^n$, and $p_0$ defined as in Lemma~\ref{2.3} b), we can define $p$ and $v$ by \eqref{(2.7)} and \eqref{rd_eq1}--\eqref{rd_eq1a}, respectively. Here, $\phi'$ is given by \eqref{eq2-9}. As we know that $\phi^n$ and $p_0$ belong to the canonical spaces by Lemma~\ref{2.3} b), we get $v\in \mE_v$ and $p\in\mE_p$ by standard results on the Stokes equation (see, e.g., \cite{Hieber-Saal18}, Section~2.6, and \cite{Pruess-Simonett16}, Section~7.2). By construction, $(v,p,\eta)$ is a solution of \eqref{resfsi}.

\subsection{Uniqueness of the solution}
\label{sec_uniqueness}
To show that the solution of \eqref{linfsi} is unique, let $(v,p,\eta)$ be a solution with zero right-hand side and zero initial data. Then the Laplace transform in $t$ and partial Fourier transform in $x'$ are well-defined, and the calculations above show, in particular, that
\[ \heta=-\frac{|\xi'|^2}{N_L(\lambda,|\xi'|)}\hf_\eta = 0 \]
for almost all $\xi'\in\R^{n-1}$. Therefore, $\eta=0$ which implies that
$(v,p)$ is the solution of the Dirichlet Stokes system with zero data.
Therefore, $v=0$ and $p=0$.

\medskip

This finishes the proof of Theorem~\ref{2.2}.

\begin{remark}
  \label{remark-finite-time}
  Theorem~\ref{2.2} was formulated on the infinite time interval $(0,\infty)$ with exponentially weighted spaces with respect to $t$. As usual in the theory of maximal regularity, we obtain the same results on finite time intervals $t\in J = (0,T)$ with $T<\infty$ without weights, i.e., with $\rho=0$. This is due to the fact that on finite time intervals the weighted and unweighted norms are equivalent and that there exists an extension operator from $(0,T)$ to $(0,\infty)$ acting on all spaces above.

  Therefore, the results of Theorem~\ref{2.2} hold  with $\rho=0$ on the finite interval $J=(0,T)$. As we consider the nonlinear equation on a finite time interval, we will replace the function spaces above by $\mE_v := H^1(J; L^p(\R^n_+))\cap L^p(J;H_p^2(\R^n_+))$ etc., keeping the same notation.
\end{remark}

\section{The nonlinear system}

To prove mapping properties of the nonlinearities we employ sharp
estimates for anisotropic function spaces
provided in \cite{koehnesaal}. In fact, we can proceed
very similar as in \cite[Section~5.2, Proposition~5.6]{koehnesaal}.
For $\omega_j\in\N_0$, $j=1,\ldots,\nu$, we define a weight
vector as $\omega:=(\omega_1,\ldots,\omega_\nu)$ and denote by
$\dot\omega:=\mathrm{lcm}\{\omega_1,\ldots,\omega_\nu\}$ the
lowest common multiple. Further, for $n=(n_1,\ldots,n_\nu)\in\N^\nu$
we write
\[
	\R^n=\R^{n_1}\times\cdots\times\R^{n_\nu}.
\]
The (generalized) Sobolev index of an $E$-valued
anisotropic function space then reads as
\[
        \frac{1}{\dot{\omega}} \left( s - \frac{\omega \cdot n}{p} \right)
	=: \left\{
        \begin{array}{rc}
                \mbox{ind}(B^{s, \omega}_{p, q}(\bR^n,\,E)), &
		s \in\R,\ 1 < p < \infty,\ 1 \leq q \leq \infty, \\[0.5em]
                \mbox{ind}(H^{s, \omega}_p(\bR^n,\,E)),      &  -\infty
< s < \infty,\ 1 < p < \infty,                        \\[0.5em]
                \mbox{ind}(W^{s, \omega}_p(\bR^n,\,E)),      &  0 \leq s
 < \infty,\ 1 \leq p < \infty,
        \end{array}
        \right.
\]
where $\omega\cdot n=\sum_{j=1}^\nu \omega_jn_j$.
Note that we have the corresponding definition, if
$\R^n$ is replaced by a cartesian product of Intervals.
For an introduction
to anisotropic spaces such as $H^{s, \omega}_{p}(\R^n,E)$
we refer to \cite{koehnesaal} and the references cited therein.
In the situation considered here
we always have $\omega=(2,1)$ and the anisotropic
spaces below can be represented as an intersection such as
\begin{align*}
	&H^{1,(2,1)}_p(J\times \R^{n-1},L^p(\R_+))\\
	&=H^{1/2}_p(J,L^p(\R^{n-1},L^p(\R_+)))
	\cap L^p(J,H^1(\R^{n-1},L^p(\R_+))),
\end{align*}
for instance. In this case we have
\[
	\mbox{ind}\bigl(H^{1,(2,1)}_p(J\times \R^{n-1},L^p(\R_+))\bigr)
	=\frac12\left(1-\frac{2+n-1}{p}\right)
	=\frac12-\frac{n+1}{2p}.
\]

Now, let $J=(0,T)$. By the {\itshape mixed derivative theorem}, see e.g.\
\cite[Lemma~4.3]{Denk-Saal-Seiler08}, we have
\[
	H^{2}_p(J,W^{1-1/p}_p(\R^{n-1}))\cap
	L^p(J,W^{5-1/p}_p(\R^{n-1}))\hook H^1(J,W^{3-1/p}(\R^{n-1})).
\]
This yields
\begin{equation}\label{est-ht}
\begin{split}
	\partial_t \eta &\in W^{5/4 - 1/4p}_p(J,L_p(\R^{n-1}))
	\cap L_p(J,W^{3 - 1/p}_p(\R^{n-1})) \\
	&\hook W^{1-1/2p}_p(J,L^p(\R^{n-1}))
	\cap L^p(J,W^{2-1/p}_p(\R^{n-1}))\\
	&=W^{2	- 1/p, (2,1)}_p(J\times \R^{n-1})
\end{split}
\end{equation}
for $\eta \in \mE_3$. Again by the mixed derivative theorem we have
\begin{equation*}
	H^{2}_p(J,W^{1-1/p}_p(\R^{n-1})) \cap L^p(J,W^{5 -
	1/p}_p(\R^{n-1})) \hookrightarrow W^{2- 1/2p}_p(J,H^1_p(\R^{n-1})),
\end{equation*}
which gives us
\begin{equation}\label{est-hx}
	\begin{split}
		\partial_j \eta &\in W^{2 - 1/2p}_p(J,L_p(\R^{n-1}))
		\cap L_p(J,W^{4 - 1/p}_p(\R^{n-1}))\\
	 	&= W^{4 - 1/p, (2,1)}_p(J \times
		\R^{n-1})
	\end{split}
\end{equation}
for $\eta \in \mE_3$ and $j = 1,\,\dots,\,n - 1$. Analogously we obtain that
\begin{equation}\label{est-hxy}
	\begin{split}
	\partial_j \partial_k \eta &\in W^{3/2 - 1/2p}_p(J,L_p(\R^{n-1}))
	\cap L_p(J,W^{3 - 1/p}_p(\R^{n-1})) \\
	&= W^{3 - 1/p, (2,1)}_p(J \times \R^{n-1})\\
	&\hook\ W^{2 - 1/p, (2,1)}_p(J \times \R^{n-1})
	\end{split}
\end{equation}
for $\eta \in \mE_3$ and $j,\,k = 1,\,\dots,\,n - 1$.

For the velocity we have
\begin{equation}\label{est-u}
	                     v  \in H^{2,(2,1)}_p(J \times \R^n_+)
			     \hookrightarrow H^{2,(2,1)}_p(J \times
			     \R^{n-1},L^p(\R_+)).
\end{equation}			
Another application of the mixed derivative theorem yields
\begin{align}
		  \partial_j v & \in H^{1, (2,1)}_p(J \times \R^n_+)
		  \hookrightarrow H^{1,(2,1)}_p(J \times
		  \R^{n-1},L_p(\R_+)),\label{est-ux} \\
	\partial_j\partial_k v & \in L_p(J \times \R^n_+)
	= L_p(J \times \R^{n-1},L_p(\R_+)),\label{est-uxy}
\end{align}
for $j,\,k = 1,\,\dots,\,n$. Taking trace this also implies
\begin{align}
	v|_{\partial\R^n_+} & \in W^{1 - 1/2p}_p(J,L_p(\R^{n-1}))
	\cap L_p(J,W^{2 - 1/p}_p(\R^{n-1}))\nonumber\\
	&= W^{2 - 1/p, (2,1)}_p(J \times \R^{n-1})\label{est-ux-tr1}\\
	\partial_j v|_{\partial\R^n_+}
	& \in W^{1/2 - 1/2p}_p(J,L_p(\R^{n-1}))
	\cap L_p(J,W^{1 - 1/p}_p(\R^{n-1}))\nonumber\\
	&= W^{1 - 1/p, (2,1)}_p(J \times \R^{n-1})\label{est-ux-tr2}
\end{align}
for $j = 1,\,\dots,\,n$.

Now, we denote by $L$ the linear operator on the left hand side
of system (\ref{tfsi}) and by $N=(F_v,G,0,0,H_\eta,0,0,0)$
its nonlinear right-hand side. Then (\ref{tfsi}) is reformulated as
\[
	L(v,p,\eta)=N(v,p,\eta)+(0,0,0,0,0,v_0,\eta_0,\eta_1).
\]
We also set
\begin{align*}
	\widetilde\mE&:=\mE_v\times\mE_p\times\mE_\eta,\\
	\widetilde\mF&:=\mF_v\times\mF_g\times\{0\}\times\{0\}
	      \times \gamma_0\mE_p
	      \times\gamma_0^t\mE_v\times\gamma_0^t\mE_\eta
	      \times\gamma_1^t\mE_\eta.
\end{align*}
The nonlinearity admits the following properties.
\begin{theorem}\label{mappropnl}
Let $p\ge(n+2)/3$. Then $N\in C^\omega(\widetilde\mE,\widetilde\mF)$,
$N(0)=0$, and we have $DN(0)=0$ for the Fr\'echet derivative of $N$.
\end{theorem}
\begin{proof}
{\itshape Mapping properties of $F_v$.}
Gathering (\ref{est-ht}), (\ref{est-hxy}), and (\ref{est-ux}) we
can estimate the term
\begin{equation*}
	(\partial_t \eta -  \Delta' \eta)\,\partial_n v,
\end{equation*}
as desired, provided the vector-valued embedding
\begin{equation}
\begin{split}
	\label{emb-fu1}
	&\underbrace{W^{2 - 1/p, (2,1)}_p(J \times
	\R^{n-1})}_{\textrm{ind}_1 = 1 - \frac{n + 2}{2p}}
	\cdot \underbrace{H^{1, (2,1)}_p(J \times
	\R^{n-1},L^p(\R_+))}_{\textrm{ind}_2 = \frac{1}{2} - \frac{n +
	1}{2p}}\\
	&\hookrightarrow \underbrace{H^{0, (2,1)}_p(J \times
		\R^{n-1},L^p(\R_+))}_{\textrm{ind} = - \frac{n + 1}{2p}}
\end{split}
\end{equation}
does hold. Applying \cite[Theorem~1.7]{koehnesaal} this readily follows
if at least one of the two indices $\mbox{ind}_1$, $\mbox{ind}_2$
is non-negative.
The strictest condition to be fulfilled by
\cite[Theorem~1.7]{koehnesaal}, however, is
$\mbox{ind}_1 + \mbox{ind}_2 \geq \mbox{ind}$ in case that
both of the indices on the left-hand-side
are negative which can occur for small $p$.
It is easily seen that this condition is equivalent to
\begin{equation}\label{reqp1}
	p \geq \frac{n + 2}{3}.
\end{equation}
For the terms
\begin{equation*}
	2 (\nabla' \eta \cdot \nabla')\,\partial_n v,
	\quad |\nabla' \eta|^2 \partial^2_n v,
	\quad (\nabla' \eta,0)^\tau\,\partial_n p
\end{equation*}
we employ (\ref{est-hx}), (\ref{est-uxy}) and the vector-valued embeddings
\begin{equation}
\begin{split}
	\label{emb-fu2}
	&\big[ \underbrace{W^{4 - 1/p, (2,1)}_p(J \times
	\R^{n-1})}_{\textrm{ind}_1 = 2 - \frac{n + 2}{2p}} \big]^m \cdot
	\underbrace{H^{0, (2,1)}_p(J \times
	\R^{n-1},L^p(\R_+))}_{\textrm{ind}_2 = - \frac{n +
	1}{2p}}\\
	&\hookrightarrow \underbrace{H^{0, (2,1)}_p(J \times \R^{n-1},
	L^p(\R_+))}_{\textrm{ind} = - \frac{n + 1}{2p}}
\end{split}
\end{equation}
for $m = 1,\,2$.
Due to \cite[Theorem~1.9]{koehnesaal} the above embeddings are valid,
provided that $\mbox{ind}_1 > 0$ or, equivalently,
\begin{equation}\label{reqp1weak}
	p > \frac{n + 2}{4}.
\end{equation}
Next, (\ref{est-u}) and (\ref{est-ux}) show that we obtain the desired
estimate of the term $(v \cdot \nabla)v$, if
\begin{equation*}
	\underbrace{H^{2,(2,1)}_p(J \times \R^n_+)}_{\textrm{ind}_1 = 1
	- \frac{n + 2}{2p}} \cdot \underbrace{H^{1, (2,1)}_p(J \times
	  \R^n_+)}_{\textrm{ind}_2 = \frac{1}{2} - \frac{n + 2}{2p}}
		\hookrightarrow \underbrace{H^{0, (2,1)}_p(J \times
		\R^n_+)}_{\textrm{ind} = - \frac{n + 2}{2p}}.
\end{equation*}
This is guaranteed by \cite[Theorem~1.7]{koehnesaal}
if $\max\,\{\,\mbox{ind}_1,\,\mbox{ind}_2\,\} \geq 0$.
Again, for small values of $p$ both of the indices on the left-hand-side can
become negative. Then \cite[Theorem~1.7]{koehnesaal}
implies the embedding above if
$\mbox{ind}_1 + \mbox{ind}_2 \geq \mbox{ind}$,
which is equivalent to (\ref{reqp1}).

Thanks to (\ref{est-hx}) and (\ref{est-ux}) the term
$(v' \cdot \nabla' \eta)\partial_nv$
can be estimated by utilizing the embedding
\begin{equation}\label{3factoremb}
\begin{split}
	&\underbrace{H^{1, (2,1)}_p(J \times
	\R^{n-1},H^1_p(\R_+))}_{\textrm{ind}_1 = \frac{1}{2} - \frac{n +
	1}{2p}}\cdot
	\underbrace{W^{4 - 1/p, (2,1)}_p(J \times
	\R^{n-1})}_{\textrm{ind}_2 = 2 - \frac{n + 2}{2p}}\\
	&\cdot
	\underbrace{H^{1, (2,1)}_p(J \times
	\R^{n-1},L^p(\R_+))}_{\textrm{ind}_3 = \frac{1}{2} - \frac{n +
	1}{2p}}
	\quad\hookrightarrow\quad \underbrace{H^{0, (2,1)}_p(J \times \R^{n-1},
	L^p(\R_+))}_{\textrm{ind} = - \frac{n + 1}{2p}}.
\end{split}
\end{equation}
Note that here we also employ
\[	
	H^{2, (2,1)}_p(J \times
	\R^{n}_+)
	\ \hook \ H^{1, (2,1)}_p(J \times
	\R^{n-1},H^1_p(\R_+))
\]
and $H^1_p(\R_+)\cdot L^p(\R_+)\hook L^p(\R_+)$ which is valid
due to the Sobolev embedding $H^1_p(\R_+)\hook L^\infty(\R_+)$ for $p>1$.
Thanks to \cite[Theorem~1.7]{koehnesaal} (\ref{3factoremb}) holds, if
$\min\,\{\,\mbox{ind}_1,\,\mbox{ind}_2,\,\mbox{ind}_3\,\} \geq 0$.
If at least one of
the three indices on the left-hand side is negative, then the
sum of the negative indices on the left hand side has to exceed
the index on the right hand side. The most restrictive constraint
hence results from
$\mbox{ind}_1 + \mbox{ind}_2 + \mbox{ind}_3 \geq \mbox{ind}$,
which is fullfilled if
\begin{equation}\label{anothconstr}
	p\ge\frac{2n+3}{6}.
\end{equation}
Consequently, by our assumptions $F_v$ has the desired mapping properties,
since (\ref{reqp1}) also yields (\ref{reqp1weak})
and (\ref{anothconstr}).

{\itshape Mapping properties of $G$.}
First we show $G(v,\eta) \in H^1_p(J,\,\dot{H}^{-1}_p(\bR^n_+))$.
Integration by parts yields
$\partial_n \in \sL(L_p(J \times
\bR^n_+),\,L_p(J,\,\dot{H}^{-1}_p(\bR^n_+)))$.
Using this property
and the fact that $\eta$ does not depend on $x_n$,
it is sufficient to estimate the terms
\begin{equation*}
	\partial_t \nabla'\eta \cdot v',
	\ \nabla' \eta \cdot \partial_t v'
\end{equation*}
in $L_p(J \times \bR^n_+)$.
Thanks to (\ref{est-ht}) and the mixed derivative theorem
we know
\[
	\partial_t \nabla'\eta \in W^{1 - 1/p, (2, 1)}_p(J \times
	\R^{n-1}).
\]
The first term can thus be estimated by the vector-valued embedding
\begin{align*}
	&\underbrace{W^{1 - 1/p, (2,1)}_p(J \times \R^{n-1})}_{\textrm{ind}_1 =
	\frac12- \frac{n + 2}{2p}} \cdot \underbrace{H^{2, (2,1)}_p(J \times
	\R^{n-1},\,L^p(\R_+))}_{\textrm{ind}_2 = 1 - \frac{n + 1}{2p}}\\
		&\hookrightarrow \underbrace{H^{0, (2,1)}_p(J \times
		\R^{n-1},\,L^p(\R_+))}_{\textrm{ind}
		= - \frac{n +1}{2p}}.
\end{align*}
According to \cite[Theorem~1.7]{koehnesaal} this embedding is again valid,
if we have $\max\,\{\,\mbox{ind}_1,\,\mbox{ind}_2\,\} \geq 0$ or if
$\mbox{ind}_1 + \mbox{ind}_2 \geq \mbox{ind}$ in case that both
indices on the left hand side are negative.
The latter condition is again equivalent to (\ref{reqp1}).

The second term may be estimated by employing (\ref{est-hx}),
the vector-valued embedding (\ref{emb-fu2}) for $m = 1$,
and $\partial_t v\in L^p(J\times\R^{n-1},L^p(\R_+))$
under constraint (\ref{reqp1weak}).

To see that also $G(v,\eta) \in L^p(J,H^1_p(\R^n_+))$,
we estimate the terms
\begin{equation*}
	\partial_j \nabla' \eta \cdot \partial_n v',
	\ \nabla'\eta \cdot \partial_j \partial_n v',
	\ \nabla'\eta \cdot \partial^2_n v',
	\qquad j = 1,\,\dots,\,n - 1,
\end{equation*}
in $L^p(J \times \R^n_+)$.
Similar as above this may be accomplished by utilizing
(\ref{est-hx}), (\ref{est-hxy}), (\ref{est-ux}), (\ref{est-uxy})
in combination with the vector-valued embeddings
(\ref{emb-fu1}), and (\ref{emb-fu2}).
Once more this is feasible if (\ref{reqp1}) holds.

{\itshape Mapping properties of $H_\eta$.}
Note that $W^{1 - 1/p, (2,1)}_p(J\times \R^{n-1}))\,\hook\, \gamma_0\mE_p$.
Hence, according to (\ref{est-hx}) and (\ref{est-ux}) we can estimate the terms
\begin{equation*}
	-\nabla'\eta \cdot \partial_n v',\
	-\nabla'\eta \cdot \nabla'v^n
\end{equation*}
as desired provided that the embedding
\begin{align*}
	&\big[ \underbrace{W^{4 - 1/p, (2,1)}_p(J \times
	\R^{n-1})}_{\textrm{ind}_1 = 2 - \frac{n + 2}{2p}} \big] \cdot
	\underbrace{W^{1 - 1/p, (2,1)}_p(J \times
	\R^{n-1}))}_{\textrm{ind}_2 = \frac{1}{2} - \frac{n + 2}{2p}}\\
	&\hookrightarrow \underbrace{W^{1 - 1/p, (2,1)}_p(J
		\times \R^{n-1}))}_{\textrm{ind}
		= \frac{1}{2} - \frac{n + 2}{2p}}
\end{align*}
is at our disposal. By \cite[Theorem~1.9]{koehnesaal} this is the case
if $\mbox{ind}_1 > 0$.
Hence, the nonlinearity $H_\eta$ has the desired mapping properties,
provided that $p>(n+2)/4$.
This, in turn, is true since (\ref{reqp1}) is satisfied.

Altogether we have proved the asserted embeddings, i.p.\ that
$N(\widetilde\mE)\subset \widetilde\mF$.
The claimed smoothness of $N$ as well as $N(0)=0$ and
$DN(0)=0$ follow obviously by the fact that
$N$ consists of polynomial nonlinearities which are of quadratic
or higher order.
\end{proof}
For a Banach space $E$ we denote by $B_E(x,r)$ the open ball in $E$ with
radius $r>0$ centered in $x\in E$.
Based on Theorems~\ref{2.2} and \ref{mappropnl} we can derive
well-posedness of (\ref{tfsi}) for small data.
For simplicity we also set
\begin{align*}
	\mE&:=\left\{(v,p,\eta)\in\mE_v\times\mE_p\times\mE_\eta;
	\ \partial_t\eta=v^n,\ v'=0\ \text{on } \partial\R^n_+\right\},\\
	\mF&:= \biggl\{(f_v,g,0,0,f_\eta,v_0,\eta_0,\eta_1)
	       \in\widetilde\mF;\
	      f_v,g,0,0,f_\eta,v_0,\eta_0,\eta_1\text{ satisfy}\biggr.\\
	   &\biggl.\qquad
	   \text{the compatibility conditions (C1)-(C4)}\biggr\}.
\end{align*}

\begin{theorem}\label{mainnonlinsys}
Let $p\ge (n+2)/3$ and $T>0$. Then there is a $\kappa=\kappa(T)>0$ such
that for $(f_v,g,0,0,f_\eta,v_0,\eta_0,\eta_1)\in B_{\widetilde\mF}(0,\kappa)$
satisfying the compatibility conditions (C2)-(C4) and
\begin{equation}\label{nonlincd}
	\div v_0=\nabla'\eta_0\cdot\partial_n v'_0+g|_{t=0}
	\quad \text{in } \dot{H}^{-1}_p(\R^n_+)
\end{equation}
there is a unique solution $(v,p,\eta)\in\mE$ of system (\ref{tfsi}).
The solution depends continuously on the data.
\end{theorem}
\begin{proof}
We pick $f:=(f_v,g,0,0,f_\eta,v_0,\eta_0,\eta_1)$ as assumed.
System (\ref{tfsi}) (including exterior forces) reads as
\begin{equation}\label{fullnlsys}
	L(v,p,\eta)=N(v,p,\eta)+f.
\end{equation}
We first have to verify that the right hand side belongs to $\mF$.
Observe that (\ref{nonlincd}) gives (C1). Hence,
by our assumptions the compatibility conditions (C1)-(C3) are
satisfied.
To see compatibility condition (C4) we have to verify that
there exists an $\eta_*\in\mE_\eta$ satisfying
$(\eta_*,\partial_t\eta_*)|_{t=0}=(\eta_0,\eta_1)$ and
\[
	(g+G(v,\eta),\partial_t\eta_*)\in
	H^1_{p}(J;\dot{H}^{-1}_{p,0}(\R^n_+))
\]
for every triple $(v,p,\eta)\in\mE$ such that
$(v,\eta,\partial_t\eta)|_{t=0}=(v_0,\eta_0,\eta_1)$.
Note that by assumption there is an extension $\eta_*\in\mE_\eta$
with the prescribed traces such that
\[
	(g,\partial_t\eta_*)\in
	H^1_{p}(J;\dot{H}^{-1}_{p,0}(\R^n_+)).
\]
Hence it suffices to prove that
\begin{equation}\label{nonlinginhm1}
	(\nabla'\eta\cdot\partial_n v',0)\in
	H^1_{p}(J;\dot{H}^{-1}_{p,0}(\R^n_+))
\end{equation}
For $\phi\in \dot{H}^1_{p}(\R^n_+)$ we observe
that thanks to $v'(x',0)=0$ we obtain
\[
	\int_{\R_+}\phi(x)\nabla'\eta(x')\cdot\partial_n v'(x)\,dx_n
	=-\int_{\R_+}\nabla'\eta(x')\cdot v'(x)\partial_n\phi(x)\,dx_n.
\]
In order to deduce (\ref{nonlinginhm1})
it hence suffices to prove that
\[
	\nabla'\eta\cdot v'
	\in H^1_{p}(J;L^p(\R^n_+)).
\]
Thanks to (\ref{est-hx}) and (\ref{est-u}) this follows from the embedding
\begin{align*}
	&W^{4-1/p,(2,1)}_p(J\times \R^{n-1})
	\cdot H^{2,(2,1)}_p(J\times \R^{n-1};L^p(\R_+))\\
	&\hook\, H^{2,(2,1)}_p(J\times \R^{n-1};L^p(\R_+))
	\,\hook\, H^1_{p}(J;L^p(\R^n_+)).
\end{align*}
Applying once again \cite[Theorem~1.9]{koehnesaal} we see that
this is fulfilled if $\mbox{ind}\bigl(W^{4-1/p,(2,1)}_p(J\times
\R^{n-1})\bigr)>0$.
This, in turn, holds if $p>(n+2)/4$ which is implied by our
assumption $p\ge(n+2)/3$.
Thus, (\ref{nonlinginhm1}) follows.

Altogether we have proved that
$(f_v,g,0,0,f_\eta,v_0,\eta_0,\eta_1)\in B_{\widetilde\mF}(0,\kappa)$
satisfying the compatibility conditions (C2)-(C4) and
(\ref{nonlincd}) implies that $N(w)+f\in\mF$ for
$w\in\overline{B_\mE(0,r)}$.
Hence, (\ref{fullnlsys}) is well-defined.

Now, we set
\[
	K(w)=L^{-1}(N(w)+f), \quad w\in\overline{B_\mE(0,r)}
\]
and prove that it is a contraction on $\overline{B_\mE(0,r)}$
for $r>0$ small enough.
Theorem~\ref{2.2} yields that $L\in\sLis(\mE,\mF)$.
This and the mean value theorem imply
\begin{align*}
	\|K(w)-K(z)\|_\mE&\le C\|N(w)-N(z)\|_\mE\\
	&\le C\sup_{v\in B_\mE(0,r)}\|DN(v)\|_{\sL(\mE,\widetilde{\mF})}
	\|w-z\|_\mE\quad (w,z\in \overline{B_\mE(0,r)}).
\end{align*}
Fixing $r>0$ such that
$\sup_{v\in B_\mE(0,r)}\|DN(v)\|_{\sL(\mE,\widetilde{\mF})}\le 1/2C$,
which is possible thanks to Theorem~\ref{mappropnl},
we see that $K$ is contractive.
The estimate above and Theorem~\ref{mappropnl} also imply
\begin{align*}
	\|K(w)\|_\mE
	&\le \|K(w)-K(0)\|_\mE+ C\|f\|_\mF\\
	&\le \frac{r}2+ C\kappa
	\quad (w\in \overline{B_\mE(0,r)}).
\end{align*}
Choosing $\kappa\le r/2C$ we see that $K$ is indeed a contraction
on $\overline{B_\mE(0,r)}$.
The contraction mapping principle gives the result.
\end{proof}
By the equivalence of the systems (\ref{fsi}) and (\ref{tfsi})
given through the diffeomorphic transform introduced in
Section~\ref{sectrans}, it is clear that Theorem~\ref{mainnonlinsys}
implies our main result Theorem~\ref{main}.

\bibliographystyle{abbrv}
\bibliography{fl_str_int}

\end{document}